\documentclass[11pt,a4paper]{amsart}

\title{Solid lines in axial algebras of Jordan type $\tfrac{1}{2}$ and Jordan algebras}
\author{Jari Desmet}
\address{\parbox{\linewidth}{Ghent University, Department of Mathematics: Algebra and Geometry \\Krijgslaan 281 -- S25, 9000 Gent, Belgium}}
\email{\href{mailto:jari.desmet@ugent.be}{jari.desmet@ugent.be}}
\date{\today}
\makeatletter
\@namedef{subjclassname@2020}{%
  \textup{2020} Mathematics Subject Classification}
\makeatother

\usepackage[english]{babel}
\usepackage{amsmath,amssymb,amsthm,mathtools}
\usepackage[utf8]{inputenc}
\usepackage[margin=3cm]{geometry}
\usepackage{hyperref}
\usepackage[capitalize,noabbrev]{cleveref}
\usepackage{todonotes}
\usepackage{enumitem}
\usepackage{verbatim}
\makeatletter
\newsavebox{\@brx}
\newcommand{\llangle}[1][]{\savebox{\@brx}{\(\m@th{#1\langle}\)}%
  \mathopen{\copy\@brx\kern-0.5\wd\@brx\usebox{\@brx}}}
\newcommand{\rrangle}[1][]{\savebox{\@brx}{\(\m@th{#1\rangle}\)}%
  \mathclose{\copy\@brx\kern-0.5\wd\@brx\usebox{\@brx}}}
\makeatother
\allowdisplaybreaks

\DeclareMathOperator{\kar}{char}

\DeclareMathOperator{\id}{id}
\DeclareMathOperator{\Spec}{Spec}

\DeclareMathOperator{\Miy}{Miy}

\newcommand{\F}{\mathbb{F}}
\newcommand{\E}{\mathbb{E}}

\newcommand{\Aut}{\mathrm{Aut}}
\newcommand{\doubleflat}{\mathfrak{J}(0)}
\newcommand{\baric}{\mathfrak{J}(1)}
\newcommand{\baricquo}{\overline{\mathfrak{J}(1)}}
\newcommand{\subalg}{\llangle a,b \rrangle }
\newcommand{\oneb}{u}

\setlist[enumerate]{label = \rm(\roman*)}

\newcommand{\fusionlaw}{\mathcal{J}(\tfrac{1}{2})}

\newtheorem{theorem}{Theorem}[section]
\newtheorem{corollary}[theorem]{Corollary}
\newtheorem{lemma}[theorem]{Lemma}
\newtheorem{proposition}[theorem]{Proposition}

\theoremstyle{definition}
\newtheorem{definition}[theorem]{Definition}

\theoremstyle{remark}

\newtheorem{remark}[theorem]{Remark}

\begin{document}
	\maketitle
	\begin{abstract}	
	 We show that a primitive axial algebra of Jordan type $ \tfrac{1}{2}$ is a Jordan algebra if and only if  every $2$-generated subalgebra is \emph{solid}, a notion introduced recently by Ilya Gorshkov, Sergey Shpectorov and Alexei Staroletov.	
 As a byproduct, we show that a subalgebra generated by axes $a,b$ is solid if and only if the associator $[L_a,L_b]$ is a derivation. Moreover, we show that $2$-generated subalgebras that are not solid contain precisely $3$ axes.
	\end{abstract}	
	\section{Introduction}
	Primitive axial algebras of Jordan type $\eta$ were introduced in 2015 by John Hall,
	Felix Rehren and Sergey Shpectorov \cite{primjordan,universalaxial}. A primitive axial algebra $(A,X)$ of Jordan type $\eta$ is a non-associative algebra $A$ required to be generated by a set of primitive idempotents $X$ the multiplication of which is diagonalizable with eigenvalues $0,1,\eta$, and the eigenvectors of which multiply following a specific fusion law (see \cref{sec:prelim}).
	
	In \cite[Conjecture 4.3]{JustinSergey}, Justin McInroy and Sergey Shpectorov conjecture that the connected (in the sense of \cite[Section 3.2]{JustinSergey}) primitive axial algebras of Jordan type $\eta = \tfrac{1}{2}$ are either Jordan algebras, or quotients of Matsuo algebras, certain algebras that are constructed from $3$-transposition groups (see \cref{def:matsuo}). For $2$-generated and $3$-generated primitive axial algebras of Jordan type, this has been shown to be true (\cite{GorshkovStaroletov}), and $3$-generated algebras are in fact all Jordan algebras. Recently, Tom De Medts, Louis Rowen and Yoav Segev studied the $4$-generated case in \cite{De_Medts_2023}, though it is still unclear whether $4$-generated axial algebras are either Jordan algebras or quotients of Matsuo algebras.
	
	To further advance the structure theory of primitive axial algebras of Jordan type $\tfrac{1}{2}$, Gorshkov, Staroletov and Shpectorov introduced the idea of \emph{solid subalgebras} (also called \emph{solid lines}) in \cite{gorshkov2024solid}. These are subalgebras of an axial algebra $A$ generated by two axes for which all primitive idempotents of this subalgebra are axes of the algebra $A$. The intuition behind this notion is that in a Jordan algebra, any $2$-generated subalgebra is automatically solid by the Peirce decomposition \cite[Section III.1, Lemma 1]{jacobson}, while for non-Jordan Matsuo algebras, most $2$-generated subalgebras contain at most $3$ axes. Examples of non-Jordan Matsuo algebras containing both solid lines and non-solid lines do exist, and an example was found by Gorshkov and Staroletov (see \cite[Example 7.2]{gorshkov2024solid}), namely $M(3^3\colon S_4)$, the Matsuo algebra constructed from the group $3^3\colon S_4$ as in \cref{def:matsuo}.
	
	We prove the following three results, the first of which extends the results in \cite{gorshkov2024solid}.
	\begin{theorem}\label{thm:solid14}
		Let $(A,X)$ be  a primitive axial algebra of Jordan type $\tfrac{1}{2}$ over a field $\mathbb{F}$ with $\kar \mathbb{F} \neq 2$. Write $(\cdot,\cdot)$ for the unique Frobenius form on $A$. Given $a,b\in X$, the subalgebra $\subalg$ is solid whenever $(a,b)\neq \tfrac{1}{4}$ or $\subalg$ is not $3$-dimensional. 
	\end{theorem}
	\begin{theorem}\label{thm:solidiffderiv}
		Let $(A,X)$ be  a primitive axial algebra of Jordan type $ \tfrac{1}{2}$ over a field $\mathbb{F}$ with $\kar \mathbb{F} \neq 2$. For every $a,b\in X$ with $a\neq b$, the subalgebra $\subalg$ is solid if and only if the associator $[L_a,L_b]$ is a derivation of $A$.
	\end{theorem}
	\begin{theorem}\label{thm:allsolidisjordan}
			Let $(A,X)$ be  a primitive axial algebra of Jordan type $\tfrac{1}{2}$ over a field $\mathbb{F}$ with $\kar \mathbb{F} \neq 2$. If $\kar\F =3$, assume moreover that $X$ linearly spans $A$. Then the subalgebra $\subalg$ is solid for all $a,b\in X$ if and only if $A$ is a Jordan algebra.
	\end{theorem}

	We give an overview of the new ideas that we needed to prove these results.
	
	The crucial step in proving \cref{thm:solid14}, which is done in \cref{sec:thm11}, is finding polynomials $P_{x,y},Q_x$ (defined in \cref{lem:polytocheck}) of small degree such that axes in $\subalg$ and roots of these polynomials correspond to each other in a subtle way. Using this, we can prove that once we have enough axes in $\subalg$, all primitive idempotents have to be axes, since then the defined polynomials would have to be identically zero. The proof given here extends the results of \cite{gorshkov2024solid} to all $2$-generated subalgebras in arbitrary characteristic, and the result is sharp in the sense that counterexamples exist in every characteristic when one of the conditions in the theorem does not hold. The method of proof in \cite{gorshkov2024solid} is computational in nature, and requires the classification of $3$-generated primitive axial algebras of Jordan type. Moreover, in \cite{gorshkov2024solid}, the techniques used for different types (see \cref{sec:2gen}) of $2$-generated subalgebras vary from case to case. We present a more conceptual and uniform method of proof, relying only on the classification of $2$-generated subalgebras.
		
		Intuitively, a primitive axial algebra of Jordan type $ \tfrac{1}{2}$ containing solid lines is equivalent to its automorphism group scheme having positive dimension. \cref{thm:solidiffderiv} quantifies this, as a subalgebra $\subalg\subseteq A$ is solid if and only if the associator defined by $D_{a,b}(x) \coloneqq [L_a,L_b](x) = a(bx)-b(ax)$ is a derivation of $A$. To prove this statement (\cref{sec:thm12}) we continue using the class of polynomials $P_{x,y},Q_x$ from the previous paragraph, but we need to use base change techniques to obtain that $D_{a,b}$ is a derivation if $\subalg$ is solid.
		
		In characteristic zero, the converse direction is obtained by taking the exponential of the derivation, but in positive characteristic other techniques are required. By studying the roots of the polynomials $P_{x,y},Q_x$ over  a bigger base ring, we obtain more information about the multiplicity of roots over the ground field. Using this, we can prove \cref{thm:solidiffderiv}.
				
		Rings $A$ with $2\in A^\times$ such that $D_{a,b}$ is a derivation for all $x,y\in A$ are called \emph{almost Jordan rings}, a concept that was first studied by Marshall Osborn \cite{almostjordan,identityfour}. 
		Now, it turns out that an almost Jordan algebra spanned by primitive idempotents will always be a Jordan algebra. We then combine \cref{thm:solid14,thm:solidiffderiv} to finally obtain \cref{thm:allsolidisjordan} by an induction argument, see \cref{sec:thm13}. 
		
		The methods used to prove \cref{thm:allsolidisjordan} allow us to interpret \cref{thm:solidiffderiv} in an interesting way. It tells us that the subspace of $D_{x,y}$ with $x,y\in A$ such that $D_{x,y}$ is a derivation, quantifies how close the algebra is to being Jordan. In contrast, subalgebras $\subalg$ for which $D_{a,b}$ is not a derivation correspond to two involutions, the product of which has order 3. This shows in a very concrete way the dichotomy between Jordan and Matsuo algebras. We believe these results give a better idea of the general structure of primitive axial algebras of Jordan type $\tfrac{1}{2}$, and will lead to more examples.
	
	\subsection*{Acknowledgements}
	The author is supported by the FWO PhD mandate 1172422N. The author is grateful to his supervisor Tom De Medts for his guidance, to Sergey Shpectorov for some very insightful discussions, and to Alexei Staroletov for expositional improvements. The author also wants to thank an anonymous referee for their thorough reading, considerably improving the quality of the paper.
	\section{Preliminaries}\label{sec:prelim}
	The first ingredient we need to define axial algebras is the concept of a \emph{fusion law}.
	\begin{definition}
	A fusion law $(\mathcal{F},\star)$ is a set $\mathcal{F}$ with a map $\star \colon \mathcal{F}\times\mathcal{F} \to 
2^{\mathcal{F}}$. Here $2^\mathcal{F}$ denotes the set of all subsets of $\mathcal{F}$.
	\end{definition}	
	A fusion law can be denoted using a table, just as as any other binary operation. If the product of $a,b\in \mathcal{F}$ is the empty set, we omit the entry. We also omit the set brackets $\{\}$ in the entries.	 The fusion law we are interested in is the Jordan fusion law $\mathcal{J}(\eta)$ (\cref{table:jordanmonster}) with parameter $\eta=\tfrac{1}{2}$.
	\begin{table}
		\[
		\begin{array}{c || c | c | c}
			\star & 1&0 &\eta \\
		\hline
		\hline
		1& 1 &  & \eta\\
		\hline
		0 & & 0 & \eta \\
		\hline
		\eta & \eta & \eta & 0,1 
		\end{array}
		\]
		\caption{The Jordan fusion law $\mathcal{J}(\eta)$}\label{table:jordanmonster} 
	\end{table}
	Using these fusion laws, we can now define \emph{axial algebras.}
	\begin{definition} Let $R$ be a unital, associative, commutative ring, let $(\mathcal{F},\star)$ be a fusion law with $\mathcal{F}\subseteq R$ and $A$ a commutative non-associative $R$-algebra.
		\begin{enumerate}
			\item For $a\in A$, let $L_a$ denote the endomorphism of $A$ defined by $L_a(x) = ax$ for all $x\in A$. If $\lambda\in R$ is an eigenvalue of $L_a$, the $\lambda$-eigenspace will be denoted by $A_\lambda(a)$.
			\item A non-zero idempotent $a\in A$ is an $\mathcal{F}$-axis if $L_a$ is semisimple, $\Spec(L_a)\subseteq \mathcal{F}$ and for all $\lambda,\mu\in \Spec(L_a)$:
			\[ A_{\lambda}(a)A_{\mu}(a) \subseteq \bigoplus_{\nu \in \lambda \star \mu} A_{\nu}(a).\]
			An $\mathcal{F}$-axis is \emph{primitive} if $A_1(a) = \langle a \rangle$.
			\item $(A,X)$ is a (primitive) $\mathcal{F}$-axial algebra if $X\subset A$ is a set of (primitive) $\mathcal{F}$-axes that generate $A$. If  $|X| = k\in \mathbb{N}$, we will call $(A,X)$ a $k$-generated $\mathcal{F}$-axial algebra.
		\end{enumerate}
	\end{definition}
	We will often write $A$ instead of $(A,X)$ for convenience, and will only consider primitive $\fusionlaw$-axial algebras. We will call these algebras \emph{primitive axial algebras of Jordan type $\tfrac{1}{2}$}.
	
	Note that we defined axial algebras over rings instead of fields. This is because we want to make some arguments which require base change techniques, see e.g.\@ \cref{prop:dualnumbers}. 
	
	A basic result in the theory of axial algebras is the Seress lemma.
	\begin{definition}[Seress property]
		A fusion law $\mathcal{F}$ is called Seress if both $0,1\in \mathcal{F}$ and for every $\lambda \in \mathcal{F}$ we have $1\star \lambda, 0\star \lambda \subseteq \{\lambda\}$.
	\end{definition}
	\begin{lemma}[Seress lemma, {\cite[Proposition 3.9]{universalaxial}}]\label{lem:seress}
		Let $A$ be an axial algebra with fusion law $\mathcal{F}$. If $\mathcal{F}$ is Seress, then every axis $a$ associates
		with $A_1(a) \oplus A_0(a)$.
	\end{lemma}
	What makes the Jordan fusion law interesting to study is its connection to group theory. 
	\begin{definition}[{\cite[Definition~2.5]{moduleaxial}}]
	\leavevmode
	\begin{enumerate}
        \item
            A fusion law $(\mathcal{F},\star)$ is called \emph{$\mathbb{Z}/2\mathbb{Z}$-graded} if $\mathcal{F}$ can be partitioned into two subsets $\mathcal{F}_+$ and $\mathcal{F}_-$ such that
            \begin{align*}
                \lambda \star \mu \subseteq \mathcal{F}_+ & \text{ whenever }  \lambda, \mu  \in \mathcal{F}_+, \\
                 \lambda \star \mu  \subseteq \mathcal{F}_+ & \text{ whenever }  \lambda , \mu  \in \mathcal{F}_-, \\
                 \lambda \star \mu  \subseteq \mathcal{F}_- & \text{ whenever } \lambda \in \mathcal{F}_+ \text{ and } \mu \in \mathcal{F}_- \text{ or } \lambda \in \mathcal{F}_- \text{ and } \mu \in \mathcal{F}_+ .
            \end{align*}
        \item
            Let $(A,X)$ be a $(\mathcal{F},\star)$-axial algebra for some $\mathbb{Z}/2\mathbb{Z}$-graded fusion law $(\mathcal{F},\star)$.
            We associate to each $(\mathcal{F},\star)$-axis $a$ of $A$ a \emph{Miyamoto involution} $\tau_a \in \Aut(A)$ defined by linearly extending
            \[
            x^{\tau_a} = \begin{cases}
                \phantom{-}x & \text{if $x \in A_{\mathcal{F}_+}(a)$}, \\
                -x & \text{if $x \in A_{\mathcal{F}_-}(a)$}.
            \end{cases}
            \]
            Since the fusion law is $\mathbb{Z}/2\mathbb{Z}$-graded, these maps define automorphisms of $A$.
        \item
            We call the subgroup $\langle \tau_e \mid e \in X \rangle \leq \Aut(A)$ the \emph{Miyamoto group} of the axial algebra $(A,X)$,
            and we denote it by $\Miy(A, X)$.
	\end{enumerate}
	\end{definition}
	Other than Jordan algebras, the most basic example of Jordan type axial algebras are so-called \emph{Matsuo algebras}. They arise from $3$-transposition groups, objects that have been studied extensively by Jonathan Hall and Hans Cuypers among others \cite{cuypershall,hallsoicher}. 
	\begin{definition}\label{def:matsuo}
	\begin{enumerate}
		\item A \emph{$3$-transposition group} $(G,D)$ is a group $G$ generated by a conjugacy class of involutions $D$ such that the product of any two elements in $D$ has order at most $3$.
		\item Given a field $\F$, a $3$-transposition group $(G,D)$ and $\eta\in \F$, we write $M_{\eta}(\F,(G,D))$ for the algebra $\F D$ with multiplication
			\[ a \cdot b \coloneqq \begin{cases}
				a &\text{ if } a=b,\\
				0 &\text{ if } o(ab) =2,\\
				\tfrac{\eta}{2}(a+b-a^{b}) &\text{ if } o(ab)=3,
			\end{cases}\]
			where $a^b \coloneqq bab\in D$ is the conjugate of $a$ by $b$ in the group $G$.
			When all parameters are clear from context, we also write $M(G)$ for $M_{\eta}(\F,(G,D))$.
		\item We call direct sums of the algebras above \emph{Matsuo algebras}.
	\end{enumerate}
	\end{definition}
	It is easy to check that Matsuo algebras $(M,D)$ with parameter $\eta$ are primitive axial algebras with respect to the fusion law $\mathcal{J}(\eta)$. In fact, whenever $\eta \neq \tfrac{1}{2}$, these are essentially the only examples of primitive $\mathcal{J}(\eta)$-axial algebras \cite[Theorem~1.3]{primjordan}. In \cite[Section~4]{moduleaxial}, it was shown that the Miyamoto group of a Matsuo algebra returns the group you started with, quotiented by a central subgroup.	
	When $\eta=\tfrac{1}{2}$, the situation is much less clear. Shpectorov, Gorshkov and Staroletov introduced the notion of solidness of $2$-generated subalgebras of a primitive $\fusionlaw$-axial algebra as a tool to help classify these algebras. For convenience, we will sometimes call $2$-generated subalgebras \emph{lines}. 
	For ease of notation, we will often write the base change of an $\F$-algebra $A$ with a unital, commutative, associative $\F$-algebra $R$ by $A_R \coloneqq A\otimes_\F R$, as is commonly done. 
	\begin{definition}\label{def:solid}
		Let $(A,X)$ be  a primitive axial algebra of Jordan type $\tfrac{1}{2}$ over a field $\F$ with $\kar \F \neq 2$. For every $a, b\in X$ with $a\neq b$ we will call the $2$-generated subalgebra $B= \subalg $  \emph{solid} if every primitive idempotent of $\subalg_\E$ satisfies the Jordan fusion law in $A_\E$, for an (and thus every) algebraic closure $\E$ of $\F$.
	\end{definition}
	Note that this definition differs from \cite[Definition~5.1]{gorshkov2024solid}. The only real difference is for the field of three elements $\F_3$, as we will see in \cref{cor:polyzeroifsolid}. In that case, \cite[Definition~5.1]{gorshkov2024solid} would consider every $2$-generated subalgebra solid by \cref{lem:nonalgtoric,lem:allprimflat,lem:allprimbaric,cor:finiteaxesflat,cor:finiteaxesbaric}, which is not desired, see \cref{thm:allsolidisjordan,rem:counterex3}.
	\begin{lemma}\label{prop:jordan-solid}
		Every $2$-generated subalgebra in a Jordan algebra $A$ is solid.
	\end{lemma}
	\begin{proof}
		Every idempotent in $A_\E$, with $\E$ an algebraic closure, satisfies the Pierce decomposition \cite[Section III.1, Lemma 1]{jacobson}, hence also the $\fusionlaw$ fusion law. So every $2$-generated subalgebra is solid.
	\end{proof}
	The last tool that we will need is the existence of a \emph{Frobenius form}.
	\begin{definition}
		A bilinear form $(\cdot,\cdot)$ on an $\F$-algebra $A$ is a \emph{Frobenius form} if
		\[ (a,b c) = (a b,c) \text{ for all $a,b,c\in A$.} \]
	\end{definition}	
	\begin{lemma}[{\cite[Theorem~4.1 and Lemma~4.3]{frobeniusform}}]\label{lem:frobform}
		Given a primitive axial algebra $(A,X)$ of Jordan type $\tfrac{1}{2}$ over a field $\F$, $\kar \F \neq 2$, there exists a unique Frobenius form $(\cdot,\cdot)$ on $A$ such that $(a,a)=1$ for every $a\in X$ and $(\cdot,\cdot)$ is invariant under $\Aut(A)$.
	\end{lemma}

	From here on, we will use $(\cdot,\cdot)$ as notation for the Frobenius form (normalized as above) of a given axial algebra $(A,X)$ of Jordan type $\tfrac{1}{2}$ over a field $\F$.
	We immediately have some basic computations which will be used throughout.
	\begin{lemma}\label{lem:basiccomp}
		Let $(A,X)$ be a primitive axial algebra of Jordan type $\tfrac{1}{2}$ over a ring $R$ with $2\in R^\times$. Let $x\in A$ and $a\in X$ an axis. Then we have
		\begin{enumerate}
			\item $ax = \tfrac{1}{4}(x-x^{\tau_a}) + (a,x)a,$
			\item $x^{\tau_a} = x+4(a,x)a - 4ax,$
			\item $a(ax) = \tfrac{1}{2}(ax+(a,x) a ).$  
		\end{enumerate}	
	\end{lemma}
	\begin{proof}
		As $a$ is an axis, we write $x=x_0+x_1+x_{1/2}$ in a unique way as a sum of eigenvectors of $L_a$, where $x_\lambda \in A_{\lambda}(a)$. Then $x^{\tau_a} = x_0 + x_1 - x_{1/2}$, so $x_{1/2} = \tfrac{1}{2}(x-x^{\tau_a})$. Since $a$ is primitive, we have $x_1 = \lambda a$ for some $\lambda\in R$, and  $(a,x) = (a,x_1) = \lambda$, thus $x = x_0 + (a,x)a + \tfrac{1}{2}(x-x^{\tau_a})$. This implies $ax = \tfrac{1}{4}(x-x^{\tau_a}) + (a,x)a$. Now (i), (ii) and (iii) follow.
	\end{proof}
	Lastly, we will need the associators alluded to in the introduction.
	\begin{definition}\label{def:associators}
		Let $A$ be a commutative algebra over a field $\F$, and $a,b\in A$. Define $D_{a,b}\colon A\to A$ by $D_{a,b}(x)= a(bx) - b(ax)$. Then $D_{a,b} = [L_a,L_b]$, the associator of $a$ and $b$.
	\end{definition}
	\section{$2$-generated subalgebras}\label{sec:2gen}
	We will give a review of $2$-generated axial algebras of Jordan type $\tfrac{1}{2}$. For more information on this topic, we refer to \cite{gorshkov2024solid,primjordan}. Any $2$-generated primitive axial algebra of Jordan type  $\tfrac{1}{2}$ is a quotient of a $3$-dimensional algebra, classified by the value of the Frobenius form $(a,b)$, where $a$ and $b$ are generating axes of the $2$-generated axial algebra $B$ \cite[Theorem~4.6]{primjordan}. We discuss this classification, and review some facts we will use.
	\begin{definition}
		Let $B=\subalg$ be a $2$-generated primitive axial algebra of Jordan type $\tfrac{1}{2}$, and denote by $(,)$ its unique Frobenius form from \cref{lem:frobform}.
		\begin{enumerate}
			\item If $(a,b) \neq 0,1$ then  we call $B$ \emph{toric},
			\item if $(a,b) =0$ then we call $B$ \emph{flat},
			\item if $(a,b)=1$ then we call $B$ \emph{baric}.
		\end{enumerate}
		Moreover, we collectively call the flat and baric algebras \emph{unipotent.}
	\end{definition}
	
	The $2$-generated baric axial algebra $B$ is in fact baric in the usual sense, i.e.\@ there exists a morphism $\mathrm{wt}_a\colon B\to \mathbb{F}$, given by the map $\mathrm{wt}_a(x) = (a,x)$ for all $x\in B$, where $a$ is any given axis in $B$. Moreover, $\mathrm{wt}_a(c) =1$ for any axis $c$ in $B$. For more on this, see \cite[p.11]{gorshkov2024solid}.

	One major distinction between toric lines and flat or baric lines is the structure of their automorphism groups. In the toric case, the identity component of the automorphism group will be a multiplicative group. In contrast, for baric or flat algebras, the identity component of their automorphism group is unipotent, hence the naming conventions.
	\subsection{Toric algebras} In this subsection, we describe toric algebras, the variety of primitive idempotents in toric algebras and the action of the Miyamoto involutions.
	\begin{lemma}\label{lem:toricfacts}
		Let $B = \subalg$ be a toric algebra over an algebraically closed field $\F$. Then $B$ is a $3$-dimensional simple unital Jordan algebra (with unit $\oneb$), and there exists a basis $\{e,\oneb,f\}$ of $B$ with $e^2=f^2=0$, $ef = \tfrac{1}{8}\oneb$, $a= e+\tfrac{1}{2}\oneb+f$ and $b = \mu e +\tfrac{1}{2}\oneb+\mu^{-1}f$ for a certain $\mu\in \F^\times$. We  have $(e,e)=(e,u)=(f,u)=(f,f)=0$ and $(e,f)= \tfrac{1}{4}$, $(u,u)=2$.
	\end{lemma}
	\begin{proof}
		See \cite[Lemma 3.7, Lemma 3.8]{gorshkov2024solid}. Note that we can always rescale $e,f$ such that $a= e+\tfrac{1}{2}\oneb+f$. The algebra $B$ is Jordan by \cite[Theorem~1.1(5)]{primjordan}. We have that $(\oneb,\oneb) = 2$ by \cite[p.\@ 12]{gorshkov2024solid}, and since the Frobenius form associates with the algebra product we also get that $(e,e) = (e^2,u) = 0$, $(f,f) = (f^2,u) = 0$ , $(e,u) = 4(e,ef) = 4(e^2,f) = 0$, $(f,u) = 4(f,ef) = 4(f^2,e) = 0$ and $(e,f) = (ef,u) = \frac{1}{8}(u,u) = \tfrac{1}{4}$. 
	\end{proof}
	\begin{definition}\label{def:primtoric}
		For $\lambda \in \F^\times$, we write $a_\lambda = \lambda e+\tfrac{1}{2}\oneb+\lambda^{-1}f$.
	\end{definition}
	\begin{lemma}[{\cite[Lemma~3.5]{gorshkov2024solid}}]\label{lem:allprimtoric}
		If $\subalg$ is a toric algebra generated by two axes $a$ and $b$, the set of
idempotents different from $0$ and $\oneb$ is the set $\{a_\lambda\vert\lambda \in \F^{\times}\}$.	
	\end{lemma}
	The next lemma is needed for \cref{cor:polyzeroifsolid}.
	\begin{lemma}\label{lem:nonalgtoric}
		Over an arbitrary field $\F$, a toric algebra $\subalg$ contains at least $|\F|-1$ primitive idempotents.
	\end{lemma}
	\begin{proof}
		By \cite[Proposition~3.4(a)]{gorshkov2024solid}, $\subalg = \F \oplus V$ is isomorphic to a Jordan algebra of a symmetric bilinear form $f$ on a vector space $V$, with basis $v_x,v_y$ \cite[p.14]{jacobson}. Note that $f$ is non-degenerate, since it is equal to half the restriction of the Frobenius form to $V$ (\cite[p.12]{gorshkov2024solid}). Then $a= \tfrac{1}{2}(1+v)$ for a certain $v= x_av_x + y_av_y,\in V$, with $f(v,v)=1$. All elements of the form $\alpha_w = \tfrac{1}{2}(1+w)$ with $w\in V$ such that $(w,w)=1$ are primitive idempotents, since $\alpha_w^2 = \alpha_w$, $\alpha_w(1-w) = 0$ and $\alpha_w w_0 = \tfrac{1}{2}w_0$ for a $w_0\in V\setminus\{0\}$ with $f(w,w_0)=0$. 
		Write $q(x,y) = f(xv_x+yv_y,xv_x+yv_y)$ for all $x,y\in \F$. Since $f$ is non-degenerate, the projective conic $C$ defined by $q(x,y) -z^2 =0$ is smooth by \cite[Proposition~22.1]{EKM}, and it has an $\F$-rational point $(x,y,z) = (x_a,y_a,1)$, hence $C$ has $|\F|+1$ points by \cite[Remark~1.3.5]{GS21}. Since the irreducible conic $C$ and the line at infinity intersect in at most $2$ points, the affine conic defined by $q(x,y) -1 = 0$ contains at least $|\F|-1$ points.
		\end{proof}
	Lastly, we need to know a bit about the action of the Miyamoto involutions. Note that every idempotent satisfies the Jordan fusion law by \cref{prop:jordan-solid}, since $\subalg$ is Jordan.
	\begin{lemma}\label{lem:miyactiontoric}
		Let $\tau_{a_\mu}$ be the Miyamoto involution corresponding to the axis $a_\mu$. Then
		\[a_\lambda^{\tau_{a_\mu}} = a_{\lambda^{-1}\mu^2} \text{ for all } \lambda,\mu \in \F^{\times}.\]
	\end{lemma}
	\begin{proof}
		By \cref{lem:basiccomp,lem:toricfacts} we have
		\begin{multline*}
			a_\lambda^{\tau_{a_\mu}} = a_\lambda +4\left(\lambda e +\tfrac{1}{2}\oneb +\lambda^{-1}f,\mu e +\tfrac{1}{2}\oneb +\mu^{-1}f\right)a_\mu -4(\lambda e+ \tfrac{1}{2}\oneb + \lambda^{-1}f)(\mu e+ \tfrac{1}{2}\oneb + \mu^{-1}f)\\
			= a_\lambda + (\lambda\mu^{-1} +2 + \lambda^{-1}\mu)a_\mu - 2\lambda e -\tfrac{1}{2}\lambda\mu^{-1}\oneb - 2a_\mu -\tfrac{1}{2}\lambda^{-1}\mu\oneb-2\lambda^{-1}f 
			\\= \lambda e +\tfrac{1}{2}\oneb + \lambda^{-1}f + \lambda e +\tfrac{1}{2}\lambda\mu^{-1}\oneb +\lambda\mu^{-2}f +\lambda^{-1}\mu^2 e +\tfrac{1}{2}\lambda^{-1}\mu\oneb +\lambda^{-1}f \\- 2\lambda e -\tfrac{1}{2}\lambda\mu^{-1}\oneb -\tfrac{1}{2}\lambda^{-1}\mu\oneb-2\lambda^{-1}f \\
			= \lambda^{-1}\mu^2 e +\tfrac{1}{2}\oneb+\lambda\mu^{-2}f = a_{\lambda^{-1}\mu^2}. \qedhere
		\end{multline*}
			\end{proof}
	\begin{corollary}\label{cor:finiteaxes}
		Let $\subalg$ be toric. Write $G = \langle \tau_a,\tau_b \rangle$. The union of orbits $O = a^{G}\cup b^{G}$ has size $n<\infty$ if and only if $\mu$ is a primitive $n$-th root of unity, where $\mu$ is as in \cref{lem:toricfacts}.
	\end{corollary}
	\begin{proof}
		Note that $a = a_1$ and $b = a_\mu$. Given $k\in \mathbb{Z}$, we get that $a^{(\tau_a\tau_b)^k} = a_{\mu^{2k}} \in O$ and $b^{(\tau_a\tau_b)^k} = a_{\mu^{2k+1}}\in O$. It is then easy to see that $O = \{a_{\mu^{k}} \in \subalg \mid k\in \mathbb{Z} \}$. For this set to be finite, it is both sufficient and necessary that $\mu^{n}=1$ for a certain $n\in \mathbb{N}$, and if $n$ is minimal with respect to this property, then $|O|=n$.
	\end{proof}
	\begin{corollary}\label{cor:order314}
		With the same notation as  \cref{cor:finiteaxes}, if $O$ has size $3$, then $(a,b)=\tfrac{1}{4}$.
	\end{corollary}
	\begin{proof}
		We know from \cref{lem:toricfacts,cor:finiteaxes} that in this case, $a = e+\tfrac{1}{2}\oneb +f, b= \mu e +\tfrac{1}{2}\oneb +\mu^{2}f$, where $\mu$ is a primitive third root of unity. Using the description of the Frobenius form in \cref{lem:toricfacts}, we get that $(a,b) = \tfrac{1}{4}(\mu+\mu^2)+\tfrac{1}{2}$. Since $\mu$ is a primitive third root of unity, we have $\mu^2+\mu =-1$, and thus $(a,b)=\tfrac{1}{4}$.
	\end{proof}
	\subsection{Flat algebras} In this subsection, we describe flat algebras, the variety of primitive idempotents in flat algebras and the action of the Miyamoto involutions.
	\begin{lemma}\label{lem:flatfacts}
		Let $B = \subalg$ be a flat algebra over a field $\F$. Then $B$ is Jordan with unit $\oneb$ and is isomorphic to either
		\begin{enumerate}
			\item a $3$-dimensional algebra, denoted by $\doubleflat$, with basis $\{a,b,ab\}$,
			\item or a $2$-dimensional quotient of this algebra, isomorphic to $\F\oplus \F$.
		\end{enumerate}
		Moreover, in both cases we have $(a,a)=(b,b)=1$ and $(a,b)=(a,ab)=(b,ab)=0$.
	\end{lemma}
	\begin{proof}
		See \cite[Theorem~3.2(a)]{gorshkov2024solid}. Note that $\subalg$ is Jordan by \cite[Theorem~1.1(5)]{primjordan}. In both cases, $(a,ab) = (a,b) = (b,ab)$ by associativity of the Frobenius form.
	\end{proof}
	\begin{definition}\label{def:primflat}
		For $\lambda \in \F$, we write $a_\lambda = a+\lambda ab$ and $b_\lambda = b+\lambda ab$.
	\end{definition}
	\begin{lemma}\label{lem:allprimflat}
		If $\subalg$ is a flat algebra generated by two axes $a$ and $b$, the set of
idempotents different from $0$ and $\oneb$ is the set $\{a_\lambda,b_\lambda\vert\lambda \in \F^{\times}\}$.	
	\end{lemma}
	\begin{proof}
	If $\subalg\cong \F \oplus \F$, the non-trivial idempotents are $a$ and $b$ and the lemma follows. The idempotents in $\doubleflat$ were calculated in the proof of \cite[Proposition~6.6]{gorshkov2024solid} with respect to the basis $a,b,\sigma$. The calculations were made for $\kar \F =0$, but remain true for $\kar \F \neq 2$. By applying the change of basis $\sigma= ab-\tfrac{1}{2}(a+b)$ (\cite[p.10]{gorshkov2024solid}), the lemma follows. 
	\end{proof}
	Again, we describe the action of the Miyamoto involution on the primitive idempotents.
	\begin{lemma}\label{lem:miyactionflat}
		Let $\tau_{a_\mu},\tau_{b_\mu}$ be the Miyamoto involutions corresponding to the axes $a_\mu,b_\mu$ respectively. Then
		\[a_\lambda^{\tau_{a_\mu}} = a_{2\mu-\lambda} \text{ and } a_\lambda^{\tau_{b_\mu}}= a_{-4 -2\mu - \lambda} \text{ for all } \lambda,\mu \in \F.\]
	\end{lemma}
	\begin{proof}
		Note that, since $(a,b)=0$, by \cref{lem:basiccomp}(iii), $ab$ is a $\tfrac{1}{2}$-eigenvector for $L_a, L_b$ and $(ab)^2=0$ by \cite[Theorem~3.2(a)]{gorshkov2024solid}. By \cref{lem:basiccomp,lem:flatfacts}, we get 
		\begin{multline*}
			a_\lambda^{\tau_{a_\mu}} = a_\lambda + 4(a_\lambda ,a_\mu)a_\mu -4a_\lambda a_\mu 
			= a + \lambda ab + 4(a +\mu ab) - 2(2a + \mu ab+ +\lambda ab + 0) \\= a + (4\mu -2\mu +\lambda - 2\lambda)ab 
			 = a_{2\mu -\lambda},
		\end{multline*}
		and
		 \begin{multline*}
		a_\lambda^{\tau_{b_\mu}} = a_\lambda + 4(a_\lambda ,b_\mu)b_\mu -4a_\lambda b_\mu 
		 = a + \lambda ab +0b_\mu - 2(2ab + \mu ab +\lambda ab + 0) \\= a + (-4 +\lambda -2\mu -2\lambda)ab 
		= a_{ -4 -2\mu-\lambda}.\qedhere
		\end{multline*}
	\end{proof}
	\begin{corollary}\label{cor:finiteaxesflat}
		Suppose $\subalg$ is flat and not $2$-dimensional. Write $G = \langle \tau_a,\tau_b \rangle$. The orbit $O = a^G$ has size $n<\infty$ if and only if $2\neq \kar \F =p >0$, and then $|O| = p$.
	\end{corollary}
	\begin{proof}
		Note that $a = a_0$ and $b = b_0$. Given $k\in \mathbb{Z}$, we get that $a^{(\tau_a\tau_b)^k} = a_{-4k} \in O$. It is then easy to see that $O = \{a_{4k} \in \subalg \mid k\in \mathbb{Z} \}$. For this set to be finite, it is both sufficient and necessary that the characteristic of the field is positive, and since $\kar \F =p \neq 2$, we then have that $|\{a_k | k\in \{0,\dots,p-1\}\}| =p$.
	\end{proof}
	\subsection{Baric algebras}In this subsection, we describe baric algebras, the variety of primitive idempotents in baric algebras and the action of the Miyamoto involutions.
	\begin{lemma}\label{lem:baricfacts}
		Let $B = \subalg$ be a baric algebra over a field $\F$. Then $B$ is Jordan and is isomorphic to either
		\begin{enumerate}
			\item a $3$-dimensional algebra, denoted by $\baric$, with basis $\{a,v=2(ab-a),v^2\}$,
			\item a $2$-dimensional quotient of $\baric$, denoted  by $\baricquo$, with basis $\{a,v=2(ab-a)\}$,
			\item or a $1$-dimensional quotient.
		\end{enumerate}
		Moreover, we have $(a,a) =1$ and  $v,v^2$ are in the radical of the Frobenius form.
	\end{lemma}
	\begin{proof}
		See \cite[Theorem~3.2(b)]{gorshkov2024solid}. Note that $\subalg$ is Jordan by \cite[Theorem~1.1(6)]{primjordan}. 
	\end{proof}
	\begin{definition}\label{def:primbaric}
		For $\lambda \in \F$, we write $a_\lambda = a+\lambda v +\lambda^2v^2$ with $v=2(ab-a)$.
	\end{definition}
	\begin{lemma}\label{lem:allprimbaric}
		If $\subalg$ is a baric algebra generated by two different axes $a$ and $b$, the set of
idempotents different from $0$ is the set $\{a_\lambda\vert\lambda \in \F^{\times}\}$.
	\end{lemma}
	\begin{proof}
	The idempotents in $B$ were calculated in the proof of \cite[Proposition~6.4]{gorshkov2024solid} with respect to the spanning set $a,b,\sigma$, where $\sigma$ is identified with $0$ if $\subalg$ is $2$-dimensional, in their notation. By \cite[Theorem 3.2(b)]{gorshkov2024solid} The calculations were made for $\kar \F =0$, but remain true for $\kar \F \neq 2$. Since, using the notation of \cite[Theorem~3.1,Theorem~3.2(b)]{gorshkov2024solid}
		\begin{equation}
			v^2=4(ab-a)^2 = 4(\tfrac{1}{2}(a+b)+\sigma - a)^2 = (b-a)^2=(b+a - (a+b+2\sigma)) = -2\sigma,\label{eq:v2annihil}
		\end{equation}
		and $b = a+2(ab-a) - 2\sigma = a +v +v^2$, we can change coordinates to obtain the result.
	\end{proof}
	We describe the action of the Miyamoto involutions on the  idempotents one final time.
	\begin{lemma}\label{lem:miyactionbaric}
		Let $\tau_{a_\mu}$ be the Miyamoto involution corresponding to the axis $a_\mu$. Then
		\[a_\lambda^{\tau_{a_\mu}} = a_{2\mu-\lambda}  \text{ for all } \lambda,\mu \in \F.\]
	\end{lemma}
	\begin{proof}
		Since $(a,b)=1$, by \cref{lem:basiccomp}~(iii), $ab-a$ is a $\tfrac{1}{2}$-eigenvector for $L_a$, and $(ab-a)^2$ is in the annihilator ideal of $B$ (\cref{eq:v2annihil} and \cite[Theorem~3.2(b)]{gorshkov2024solid}). By \cref{lem:basiccomp}
		\begin{multline*}
			a_\lambda^{\tau_{a_\mu}} = a_\lambda + 4(a_\lambda ,a_\mu)a_\mu -4a_\lambda a_\mu  
			=a+\lambda v + \lambda^2v^2 +4(a+\mu v + \mu^2v^2 ) - 2(2a+\lambda v +\mu v +2\lambda\mu v^2) \\= a + (2\mu - \lambda)v + (\lambda^2 + 4\mu^2-4\lambda\mu)v^2 
			= a_{2\mu-\lambda}.\qedhere
		\end{multline*}
	\end{proof}
	\begin{corollary}\label{cor:finiteaxesbaric}
		Suppose $\subalg$ is baric and not $1$-dimensional. Write $G = \langle \tau_a,\tau_b \rangle$. The orbit $O = a^G$ has size $n<\infty$ if and only if $2\neq \kar \F =p >0$, and then $|O| = p$.
	\end{corollary}
	\begin{proof}
		Note that $a = a_0$ and $b = a_1$. Given $k\in \mathbb{Z}$, we get that $a^{(\tau_a\tau_b)^k} = a_{2k} \in O$. It is then easy to see that $O = \{a_{2k} \in \subalg \mid k\in \mathbb{Z} \}$. For this set to be finite, it is both sufficient and necessary that the characteristic of the field is positive, and since $\kar \F =p \neq 2$, we then have that $|\{a_k | k\in \{0,\dots,p-1\}\}| =p$.
	\end{proof}
	\section{Proof of \cref{thm:solid14}}\label{sec:thm11}
	To prove \cref{thm:solid14}, we find polynomials of small degree that determine when a subalgebra is solid (\cref{lem:polytocheck}), and then prove they have many zeroes in most cases, which forces the polynomials to be identically zero. 	
	\begin{lemma}\label{lem:polytocheck}
		Let $(A,X)$ be a primitive axial algebra of Jordan type $\tfrac{1}{2}$ over an algebraically closed field $\F$, and $a,b\in X$. The subalgebra $\subalg$ is solid if and only if
		\[ Q_{x}(c) \coloneqq c (c x) -\tfrac{1}{2}(c x + (c,x)c) \]
		and 
		\[ P_{x,y}(c) \coloneqq 4(c x)(c y)-(c,y)c x -(c y)x -(c ,x)c y  -(c x)y  -(c,xy)c+ c(xy) \]
		are zero for all primitive idempotents $c\in \subalg$ and $x,y\in A$.
		Conversely, if $c\in \subalg$ satisfies the Jordan fusion law, then $Q_x(c) = P_{x,y}(c)=0$ for all $x,y\in A$.
	\end{lemma}
	\begin{proof}
		Given a primitive idempotent $c\in \subalg$, let $\phi_c \colon A\to A$ be the map defined by $x^{\phi_c} = x+4(c,x)c - 4cx$ for $x\in A$. This is clearly an $\F$-linear map. 
	
	Suppose $P_{x,y}(c) = 0 = Q_{x}(c)$ for all $x,y\in A, c \in \subalg$. We will show $\phi_c$ is an automorphism of $A$ for every primitive idempotent $c\in \subalg$. For every $x,y\in A$ we have
		\begin{align}
			x^{\phi_c}y^{\phi_c} - (xy)^{\phi_c} &= (x+4(c ,x)c -4c x)(y+4(c ,y)c -4c y)- (xy+4(c ,xy)c -4c (xy))\nonumber\\
			&= xy + 4(c ,y)c x - 4(c y)x +4(c ,x)c y+16(c ,x)(c ,y)c -16(c ,x)c (c y)\nonumber\\
			&\hspace{3ex}-4(c x)y - 16(c ,y)c (c x) +16(c x)(c y)- (xy+4(c ,xy)c -4c (xy)) \label{eq:polytocheck}
		\end{align}
		Since $Q_x(c)=Q_y(c) = 0$, the last line is equal to $4P_{x,y}(c)=0$. As $Q_x(c)=0$ for all $x\in A, c \in \subalg$ and $(c,c)=1$ by \cref{prop:jordan-solid,lem:frobform,lem:toricfacts,lem:flatfacts,lem:baricfacts}, we have  \begin{multline*}
				x^{(\phi_c)^2} = x^{\phi_c} + 4(c,x)c^{\phi_c} - 4cx -16(c,cx)c+16c(cx) \\=x+8(c,x)c -8cx -16(c,x)c+8cx +8(c,x)c=x, 
		\end{multline*} implying that $\phi_c$ is bijective. This means $\phi_c$ is an automorphism for every primitive idempotent $c \in \subalg$. Since $\subalg$ is Jordan, by \cref{prop:jordan-solid}, $c$ is an axis of $\subalg$ and the restriction of $\phi_c$ to $\subalg$ is equal to the Miyamoto involution $\tau_c$ by \cref{lem:basiccomp}, so the set $\{a,b\}^{\left\langle \phi_c \mid c\text{ is a primitive idempotent in } \subalg \right\rangle}$ is equal to the set of all primitive idempotents in $\subalg$ by \cref{lem:miyactiontoric,lem:miyactionbaric,lem:miyactionflat,lem:allprimbaric,lem:allprimtoric,lem:allprimflat}, which means the line is solid.
		
		Conversely, if $c$ is an axis, then $Q_x(c)=0$ for every $x \in A$ by \cref{lem:basiccomp}, and then \eqref{eq:polytocheck} shows $P_{x,y}(c)=0$ for every $x,y\in A$, since $\phi_c = \tau_{c}$ is an automorphism of $A$.
	\end{proof}
	\subsection{The toric case} 
	Recall the notation of $e,\oneb,f$ as in \cref{lem:toricfacts}, chosen with respect to the generating axes $a,b$ of $\subalg$. 
	\begin{definition}\label{def:polytoric}
	Let $(A,X)$ be a primitive axial algebra of Jordan type $\tfrac{1}{2}$ over an algebraically closed field $\F$, and $a,b\in X$ such that $\subalg$ is toric.
		For $x,y\in A$, write $P^t_{x,y}$ and $Q^t_x$ for the maps
		 \[ P^t_{x,y} \colon \F^\times \to A \colon \lambda \mapsto P_{x,y}(a_\lambda) \text{ and } Q^t_{x} \colon \F^\times \to A \colon \lambda \mapsto Q_{x}(a_\lambda). \]
	\end{definition}
	After identifying $A$ with the affine space $\F^{\dim A}$ by a choice of basis, we can identify the maps $P^t_{x,y},Q^t_{x}$ with (vector-valued) rational maps, i.e.\@ maps such that the coordinate functions are all rational functions, with a double pole at zero, since $a_\lambda = \lambda^{-1}(\lambda^2 e + \lambda\tfrac{1}{2}\oneb + f)$ has a single pole at zero (see \cref{lem:polytocheck}). Multiplying $P^t_{x,y},Q^t_{x}\in A(\lambda)$ with the polynomial $\lambda \mapsto \lambda^2$, written as $\lambda^2P^t_{x,y},\lambda^2Q^t_{x}$, we can identify these maps with (vector-valued) polynomial maps. We will say a value $\mu\in \F$ is a zero for a polynomial map $P\in A[\lambda]$ if $P(\mu) = 0 \in A$. Equivalently, $\mu\in \F$ is a zero for $P\in A[\lambda]$ when it is a zero for all coordinate functions. The degree of a vector-valued polynomial map is defined in the obvious sense. As in the $1$-dimensional case, when the amount of zeroes of a polynomial $P\in A[\lambda]$, counted with multiplicity, exceeds its degree $\deg P$, then $P$ equals the zero polynomial. We will exploit this fact to prove \cref{prop:autotoric,prop:autoflat}.
	\begin{proposition}\label{prop:autotoric}
		Let $(A,X)$ be a primitive axial algebra of Jordan type $\tfrac{1}{2}$ over a field $\mathbb{F}$ and $a,b\in X$ with $(a,b)\neq 0,1,\tfrac{1}{4}$. Then $\subalg$ is solid.
	\end{proposition}
	\begin{proof}
		By \cref{def:solid}, assume $\F$ is algebraically closed. Since $(a,b)\neq 0,1,\tfrac{1}{4}$, we know $\subalg$ is toric and $\subalg$ contains at least $4$ axes by \cref{cor:order314}. We write
		\begin{multline}
			\lambda^2P^t_{x,y}(\lambda) = 4( (\lambda^2e + \tfrac{\lambda}{2}u + f)x)((\lambda^2e + \tfrac{\lambda}{2}u + f) y)\\ -(\lambda^2e + \tfrac{\lambda}{2}u + f,y)(\lambda^2e + \tfrac{\lambda}{2}u + f) x-((\lambda^3e + \tfrac{\lambda^2}{2}u + \lambda f) y)x -(\lambda^2e + \tfrac{\lambda}{2}u + f ,x)(\lambda^2e + \tfrac{\lambda}{2}u + f) y   \\-((\lambda^3e + \tfrac{\lambda^2}{2}u + \lambda f) x)y -(\lambda^2e + \tfrac{\lambda}{2}u + f,xy)(\lambda^2e + \tfrac{\lambda}{2}u + f)+ (\lambda^3e + \tfrac{\lambda^2}{2}u + \lambda f)(xy)\\
			= p_0(x,y) + p_1(x,y)\lambda + p_2(x,y)\lambda^2+p_3(x,y)\lambda^3+p_4(x,y)\lambda^4 \label{eq:onecomp}
		\end{multline}
		and
		\begin{multline}
			\lambda^2Q^t_{x}(\lambda) = (\lambda^2e + \tfrac{\lambda}{2}u + f) ((\lambda^2e + \tfrac{\lambda}{2}u + f) x) \\-\tfrac{1}{2}((\lambda^3e + \tfrac{\lambda^2}{2}u + \lambda f) x + (\lambda^2e + \tfrac{\lambda}{2}u + f,x)(\lambda^2e + \tfrac{\lambda}{2}u + f)) \\
			= q_0(x) + q_1(x)\lambda + q_2(x)\lambda^2+q_3(x)\lambda^3+q_4(x)\lambda^4 \label{eq:twocomp}
		\end{multline}
		for every $x,y\in A$.
		If $\subalg$ has at least $5$ different axes, then  $\lambda^2P^t_{x,y},\lambda^2Q^t_x$ have at least $5$ different zeroes by \cref{lem:polytocheck} while their degree is at most $4$, hence are zero. So for every $\lambda\in \F^{\times}$ and $x,y\in A$, we have that $P^t_{x,y}(\lambda) = Q^t_x(\lambda)=0$, so by \cref{lem:polytocheck},  $\subalg$ is solid.
		
		Now suppose $\subalg$ contains only $4$ axes. Then these have to be of the form $a_{\mu}$, with $\mu^4=1$ by \cref{cor:finiteaxes}. From \eqref{eq:onecomp}, \eqref{eq:twocomp} we have
		\begin{align*}
			p_4(x,y) &= 4(ex)(ey) - (e,x)ey - (e,y)ex - (e,xy)e,\\
			\begin{split}
				p_3(x,y) &= 2(ex)(uy)+2(ux)(ey)-\tfrac{1}{2}(e,y)ux-\tfrac{1}{2}(u,y)ex\\ &\hspace{3ex}-(ey)x-\tfrac{1}{2}(e,x)uy-\tfrac{1}{2}(u,x)ey-(ex)y-\tfrac{1}{2}(e,xy)u-\tfrac{1}{2}(u,xy)e+e(xy),
			\end{split}	\\
			p_0(x,y) &= 4(fx)(fy) - (f,x)fy - (f,y)fx - (f,xy)f,\\
			q_4(x) &= e(ex)-\tfrac{1}{2}(e,x)e,\\
			q_3(x)&=\tfrac{1}{2}e(ux)+\tfrac{1}{2}u(ex) - \tfrac{1}{2}ex- \tfrac{1}{4}(e,x)u- \tfrac{1}{4}(u,x)e.
		\end{align*}
		Note that $p_3(x,e) = 2q_4(x)$ and $p_3(x,u) = 4q_3(x) +e(ux)-u(ex)$.
		For all $x,y \in A$, the polynomial maps $\lambda^2P^t_{x,y},\lambda^2Q^t_x$ have the fourth roots of unity as zeroes by \cref{cor:finiteaxes,lem:polytocheck}. Any polynomial of degree at most $4$ with fourth roots of unity as roots is a scalar multiple of $\lambda^4-1$, so for all $x,y\in A$ we have $\lambda^2P_{x,y} = \alpha(x,y)(\lambda^4-1)$, $\lambda^2Q_x = \beta(x)(\lambda^4-1)$ for some $\alpha(x,y),\beta(x)\in A$. So $p_3(x,y) = 0$ for all $x,y\in A$, and thus $q_4(x)=\beta(x) = 0 $ for all $x\in A$. But this  implies $Q^t_{x}(\lambda) = 0$ for all $\lambda\in \F^{\times}$ and all $x\in A$. 		
		
		We still need to prove that $p_4(x,y) =0$ for all $x,y\in A$. Since both $p_3(x,u)=0$ and $q_3(x)=0$, we get $e(ux)=u(ex)$ for all $x\in A$. Again using $q_3(x)=0$, we also have 
		\begin{equation} e(ux)= u(ex) = \tfrac{1}{2}ex +\tfrac{1}{4}((e,x)u+(u,x)e) \label{eq:blubl}
		\end{equation} for all $x \in A$. Given $x\in A$, using that $q_4(x) = 0$ and \eqref{eq:blubl} for the third equality, we get
		\begin{multline*}
			e(x^{\tau_a}) = e(x + 4(a,x)a -4ax)
			= ex +2(a,x)e+\tfrac{1}{2}(a,x)\oneb -4e(ex)-2e(\oneb x)-4e(fx) \\
			\begin{split}
				&= ex+2(e,x)e+(\oneb,x)e+2(f,x)e +\tfrac{1}{2}(e,x)\oneb+\tfrac{1}{4}(\oneb,x)\oneb \\
			 &\hspace{3ex}+\tfrac{1}{2}(f,x)\oneb -2(e,x)e-ex -\tfrac{1}{2}(e,x)\oneb -\tfrac{1}{2}(\oneb,x)e-4e(fx)
			\end{split} \\
			=-4e(fx) +\tfrac{1}{2}(\oneb,x)e+2(f,x)e+\tfrac{1}{4}(\oneb,x)\oneb + \tfrac{1}{2}(f,x)\oneb,
		\end{multline*}
		proving that 
		\begin{equation}
			e(fx) = -\tfrac{1}{4}e(x^{\tau_a})+\tfrac{1}{2}(f,x)e +\tfrac{1}{8}(\oneb,x)e+\tfrac{1}{8}(f,x)\oneb+\tfrac{1}{16}(\oneb,x)\oneb.\label{eq:efx}
		\end{equation}
		Next we show that $p_4(fx,y) = 0$, using that $p_4(fx,y) = \alpha(fx,y) = -p_0(fx,y)$ and $2f(fx) =(f,x)f$ by the symmetric roles of $e$ and $f$, for all $x,y\in A$. We have
		\begin{multline*}
			p_4(fx,y) = -p_0(fx,y) = -4(f(fx))(fy) + (f,fx)fy + (f,y)f(fx) + (f,(fx)y)f \\ 
			=-2(f,x)f(fy) + \tfrac{1}{2}(f,y)(f,x)f+\tfrac{1}{2}(f,y)(f,x)f
			= - (f,y)(f,x)f+(f,y)(f,x)f=0.
		\end{multline*}
		On the other hand, by applying \eqref{eq:efx} three times for the third equality and using $q_4(x) = 0$ and \eqref{eq:blubl} for the fifth equality, we get
		\begin{multline*}
			p_4(fx,y) = 4(e(fx))(ey) - (e,fx)ey - (e,y)e(fx) - (e,(fx)y)e \\
			= 4(e(fx))(ey) - \tfrac{1}{8}(\oneb,x)ey - (e,y)e(fx) - (e(fx),y)e \\
			\begin{split}
				&= -(e(x^{\tau_a}))(ey) +\tfrac{1}{2}(\oneb,x)e(ey) +2(f,x)e(ey)+\tfrac{1}{4}(\oneb,x)\oneb(ey)+\tfrac{1}{2}(f,x)\oneb(ey)-\tfrac{1}{8}(\oneb ,x)ey \\ 
			&\hspace{3ex}+\tfrac{1}{4}(e,y)e(x^{\tau_a})-\tfrac{1}{8}(e,y)(u,x)e-\tfrac{1}{2}(e,y)(f,x)e-\tfrac{1}{16}(e,y)(u,x)u-\tfrac{1}{8}(e,y)(f,x)u \\
			&\hspace{3ex}+\tfrac{1}{4}(e(x^{\tau_a}),y)e-\tfrac{1}{8}(u,x)(e,y)e-\tfrac{1}{2}(f,x)(y,e)e-\tfrac{1}{16}(u,x)(u,y)e-\tfrac{1}{8}(f,x)(u,y)e 
			\end{split}\\
			\begin{split}
				&=-(e(x^{\tau_a}))(ey) + \tfrac{1}{4}(e,y)e(x^{\tau_a})+\tfrac{1}{4}(e,x^{\tau_a}y)e + \tfrac{1}{4}(f,x)ey\\&\hspace{3ex} +\big(\tfrac{1}{2}(u,x) +2(f,x)\big)\big(e(ey)-\tfrac{1}{2}(e,y)e\big) \\ 
				&\hspace{3ex}+\big(\tfrac{1}{4}(u,x) +\tfrac{1}{2}(f,x)\big)\big(u(ey) - \tfrac{1}{2}ey -\tfrac{1}{4}((e,y)u+(u,y)e)\big)
			\end{split}\\
			= -(e(x^{\tau_a}))(ey) + \tfrac{1}{4}(e,y)e(x^{\tau_a})+\tfrac{1}{4}(e,x^{\tau_a}y)e + \tfrac{1}{4}(e,x^{\tau_a})ey= -\tfrac{1}{4}p_4(x^{\tau_a},y).
		\end{multline*}
		This implies $p_4(x^{\tau_a},y)=0$ for all $x,y\in A$. Since $\tau_a$ is a bijection, we get $p_{4}(x,y)=0$ for all $x,y\in A$. This in turn implies that $\lambda^2P^t_{x,y}$ is the zero polynomial, so $P^t_{x,y}(\lambda)=0$ for all $\lambda\in \F^{\times}$. We can now use \cref{lem:polytocheck} to show that $\subalg$ is solid.		
	\end{proof}
	\begin{remark}\label{rem:counterex}
		When $\kar \F \neq 3$, subalgebras $\subalg$ with $(a,b) =\tfrac{1}{4}$ are not always solid. If that were the case, any Matsuo algebra  would be a Jordan algebra by \cref{thm:allsolidisjordan,prop:autoflat}, since in a Matsuo algebra $M(G,D)$, the value $(a,b)$ is either $0$ or $\tfrac{1}{4}$ for all $a,b\in D$. But this would contradict \cite{felixtom}. The smallest example of an axial algebra containing non-solid lines in every characteristic is $M(W(D_4))$, the $12$-dimensional Matsuo algebra coming from the Weyl group of a root system of type $D_4$.
	\end{remark}
	\subsection{The unipotent case}
	In \cite[Section 6]{gorshkov2024solid} it was proven that $\subalg$ is solid for $(a,b)=0,1$ when $\kar \mathbb{F} = 0$. We extend this result to positive characteristic. 
	\begin{definition}\label{def:polyunip}
	Let $(A,X)$ be a primitive axial algebra of Jordan type $\tfrac{1}{2}$ over $\F$, and $a,b\in X$ such that $\subalg$ is unipotent. 		Given $x,y\in A$, write $P^u_{x,y}$ and $Q^u_x$ for the maps
		 \[ P^u_{x,y} \colon \F\to A \colon \lambda \mapsto P_{x,y}(a_\lambda) \text{ and } Q^u_{x} \colon \F\to A \colon \lambda \mapsto Q_{x}(a_\lambda).\]
	\end{definition}
	Recall the notation for the different $2$-generated algebras from \cref{lem:flatfacts,lem:baricfacts}.
	\begin{proposition}\label{prop:autoflat}
		Let $(A,X)$ be  a primitive axial algebra of Jordan type $\tfrac{1}{2}$ over a field $\F$ and $a,b\in X$ such that $\subalg \not\cong \baric$ is  unipotent, or $\kar \F \neq 3 $ and $\subalg \cong \baric$. Then $\subalg$ is solid. 
	\end{proposition} 
	\begin{proof}
		By \cref{def:solid}, assume $\F$ is algebraically closed. Suppose first that $\subalg \not\cong \baric$. If $\subalg$ is isomorphic to $\F\oplus\F$, the proposition is true since the only primitive idempotents in $\subalg$ are $a$ and $b$. Suppose $\subalg \cong \doubleflat$ or $\baricquo$. Then $\subalg$ contains at least $3$ different axes $a_\lambda$ by \cref{cor:finiteaxesflat,cor:finiteaxesbaric}, so the polynomial maps $P^u_{x,y}(\lambda)$ and $Q^u_{x}(\lambda)$ from \cref{lem:polytocheck} have at least $3$ different roots for $x,y\in A$. However, $\deg P^u_{x,y} , \deg Q^u_{x}\leq 2$, since $\deg_\lambda  a_\lambda =1$ in this case, so this is only possible when $ P^u_{x,y},Q^u_{x}$ are zero for every $x,y \in A$. For flat lines, we also have that $P_{x,y}(b- \lambda ab),Q_{x}(b - \lambda ab)$ are zero for all $x,y\in A, \lambda\in \F$ by the symmetric roles of $a$ and $b$.  The result now follows from \cref{lem:polytocheck}.
		
		If $\subalg \cong \baric$ and $\kar \F\neq 3$, then the polynomial equations $P^u_{x,y}(\lambda)$ and $Q^u_{x}(\lambda)$ from \cref{lem:polytocheck} have at least $5$ different zeroes for $x,y\in A$ by \cref{lem:polytocheck,cor:finiteaxesbaric}, and $\deg P^u_{x,y} , \deg Q^u_{x}\leq 4$. Then the result again follows from \cref{lem:polytocheck}.
	\end{proof}
	\begin{remark}\label{rem:counterex3}
		Over fields of characteristic $3$, subalgebras isomorphic to $\baric$ are not always solid. If that were the case, any Matsuo algebra in characteristic 3 would be a Jordan algebra by \cref{thm:allsolidisjordan}. But this would contradict \cite{felixtom,char3}, just as in \cref{rem:counterex}.
	\end{remark}
	\cref{prop:autotoric,prop:autoflat} together show that \cref{thm:solid14} holds.
	\section{Proof of \cref{thm:solidiffderiv}}\label{sec:thm12}
	\begin{corollary}\label{cor:polyzeroifsolid}
		Let $|\F| >3$ and $\E$ an algebraic closure of $\F$. If every primitive idempotent in $\subalg$ satisfies the Jordan fusion law, then the rational maps from \cref{def:polytoric} (or \cref{def:polyunip} if $\subalg$ is unipotent), defined with respect to the $2$-generated subalgebra $\subalg_E$ of $(A_\E ,X)$,  are identically zero and $\subalg$ is solid.
	\end{corollary}
	\begin{proof}
		If $\subalg$ is toric, it contains $|\F|-1\geq 4$ axes or more by \cref{lem:nonalgtoric}, so $\subalg_E$ also contains at least $4$ axes, since axes remain axes after scalar extension. We can then repeat the argument in the proof of \cref{prop:autotoric} to show $P_{x,y}^t,Q_{x}^t$ are zero and $\subalg$ is solid.  If $\subalg$ is unipotent, it contains $|\F|\geq 5$ axes or more by \cref{lem:allprimflat,lem:allprimbaric}. Then $\subalg_E$ also contains at least $5$ axes, so the maps $P_{x,y}^u,Q_{x}^u$ have more zeroes than their degree, and thus are zero. Then by \cref{lem:polytocheck}, $\subalg$ is solid.	
		\end{proof}
 From here onwards, ideas from the theory of affine group schemes are used, but since we only use basic techniques, we do not use the technical language. For an introduction to this theory, see \cite{Milne,waterhouse}. In what follows, we will allow the value of $\lambda$ in the definitions of $a_\lambda,b_\lambda$  from \cref{def:primtoric,def:primflat,def:primbaric} to be contained in a unital, commutative, associative $\F$-algebra $R$ (or $R^\times$ when $\subalg$ is toric, where $R^\times$ denotes the invertible elements of $R$). In this case, $a_\lambda,b_\lambda$ are still idempotents in the algebra $\subalg _R$.
	\begin{remark}\label{rem:miyactionR}
		Note that the computations for \cref{lem:miyactiontoric,lem:miyactionflat,lem:miyactionbaric} remain true, even when $\lambda,\mu$ are contained in an $\F$-algebra $R$. We will use these lemmas in what follows.
	\end{remark}
	\begin{proposition}\label{prop:dualnumbers}
		Let $\F$ be algebraically closed. Given $a,b\in X$ such that $\subalg$ is solid and $\lambda\in R$, respectively $R^\times$ when $\subalg$ is toric, where $R$ is an $\F$-algebra. Then $a_{2\lambda}$, respectively $a_{\lambda^2}$ is an axis of $A_R$.
	\end{proposition}
	\begin{proof}
		The  maps $P^t_{x,y},Q^t_x,P^u_{x,y},Q^u_x$ are zero by \cref{cor:polyzeroifsolid}, so they are zero over $R$ as well. As in the proof of \cref{lem:polytocheck}, the map $\phi_{a_\lambda}$, defined by $x^{\phi_{a_\lambda}} = x + 4(x,a_\lambda)a_\lambda - 4xa_\lambda$ for all $x\in A_R$, is an automorphism. By \cref{rem:miyactionR}, applying $\phi_{a_\lambda}$ to $a$ gives the result.
		\end{proof}
	In the next proposition, $\F[\varepsilon]$ is the algebra of dual numbers, i.e.\@ $\F[\varepsilon]\coloneqq \{ a+\varepsilon b \mid a,b\in \F \}$ with $\varepsilon^2=0$. This proposition is  a common way to construct the Lie algebra of $\Aut(A)$.
	\begin{proposition}\label{prop:dualderiv}
		Let $A$ be an $\F$-algebra, and $\phi\colon A_{\F[\varepsilon]}\to  A_{\F[\varepsilon]} $ an $\F[\varepsilon]$-linear map such that there exists $f\colon A\to A$ with $x^\phi = x+\varepsilon f(x)$ for all $x\in A$. Then $f$ is a derivation of $A$ if and only if $\phi$ is an automorphism of $A_{\F[\varepsilon]}$.
	\end{proposition}
	\begin{proof}
		Note that $\phi$ is bijective, since $(x+\varepsilon y)^\phi = x + \varepsilon(f(x) + y)=0$ if and only if $x=y=0$, and $(x+\varepsilon(y-f(x)))^\phi = x + \varepsilon y$ for all $x,y\in A$.
		For $x,y\in A\subset A_{\F[\varepsilon]}$ we have $(xy)^\phi = x^{\phi}y^{\phi} = xy + \varepsilon( f(x)y + xf(y))$ if and only if $f(x)y + xf(y) =f(xy)$. Every element of $A_{\F[\varepsilon]}$ is a linear combination of elements in $A$, hence the result follows.	\end{proof}
	\begin{corollary}\label{cor:derall}
		If $\subalg$ is solid, then $D_{a,b}$ is a derivation.
	\end{corollary}
	\begin{proof}
	We may assume $\F$ is algebraically closed, since $D_{a,b}$ is a derivation over $\F$ if and only if it is a derivation over an algebraic closure. By \cref{prop:dualnumbers}, $a_\lambda$ is an axis of $A_{\F[\varepsilon]}$, where $\lambda = (1+\tfrac{1}{2}\varepsilon)^2 = 1+\varepsilon$ if $\subalg$ is toric and $\lambda= 2(\tfrac{1}{2}\varepsilon) = \varepsilon$ if $\subalg$ is unipotent. Then $a_\lambda = a+\varepsilon v$, where $v= e-f,ab$ or $2(ab-a)$ if $\subalg$ is toric, flat or baric, respectively. In all three cases, we have $a_\lambda = a + \varepsilon v$ where $v$ is a $\tfrac{1}{2}$-eigenvector of $L_a$ (for the toric case, this follows from $b -b^{\tau_a} = (\mu-\mu^{-1})(e-f)$ by \cref{lem:miyactiontoric,lem:basiccomp}, and for the unipotent case by \cref{lem:basiccomp}). By \cref{lem:basiccomp} we have
			\begin{multline}
				a(vx) = \tfrac{1}{4}(vx+vx^{\tau_a}) + \tfrac{1}{2}(v,x)a =\tfrac{1}{4}(vx +vx +2(a,x)v -4v(ax) )+\tfrac{1}{2}(v,x)a \\
					= -v(ax) +\tfrac{1}{2}(vx + (v,x)a+(a,x)v),\label{eq:anonref}
			\end{multline}
			so $[L_a,L_v](x) = -2v(ax) +\tfrac{1}{2}(vx + (v,x)a+(a,x)v)$. We compute the action of $\rho = \tau_a\tau_{a_\lambda}$ on $A$, where $\lambda\in \F[\varepsilon]$ is as above.
			Given $x\in A \subset A_{\F[\varepsilon]}$, by \cref{lem:basiccomp} and \cref{eq:anonref},
			\begin{multline}
				x^{\rho} = x^{\tau_a} + 4(a+\varepsilon v ,x^{\tau_a}) (a+\varepsilon v) - 4(a+\varepsilon v)x^{\tau_a} \\
					= x^{\tau_a} + 4(a,x^{\tau_a})a-4ax^{\tau_a} +4\varepsilon( -(v,x)a + (a,x)v - v(x+4(a,x)a -4ax ))\\
				= x + 4\varepsilon( -(v,x)a -(a,x)v - vx +4v(ax) ) )
				= x - 8\varepsilon[L_a,L_v](x).\label{eq:derivifaut}
			\end{multline}
			Now \cref{prop:dualderiv} shows that $[L_a,L_v]$ is a derivation. For each of these cases, $\subalg_{\tfrac{1}{2}}(a) = \langle v \rangle$ (\cite[Theorem~4.7(a)]{primjordan}). This implies, by the Seress Lemma (\cref{lem:seress}), that $[L_a,L_b]= \mu [L_a,L_v]$ for a certain $\mu \in \F $. So $D_{a,b} =[L_a,L_b]$ is also a derivation.
	\end{proof}
	We now prove a converse to \cref{cor:derall}. 
	When the characteristic is zero, this is easily done by taking the exponential of the derivations. In positive characteristic however, this is harder. For that reason, we will prove this using a different technique.
	\begin{lemma}\label{lem:multiplicity2}
		Suppose $P\in A[\lambda]$ is a polynomial such that $P(\mu+\nu\varepsilon)=0 \in A\otimes_{\F} \F[\varepsilon]$ for certain $\mu \in \F, \nu \in \F^{\times}$. Then $\mu$ is a root of multiplicity at least $2$ of $P$. 
	\end{lemma}	
	\begin{proof}
		Write $P(\lambda) = a_n\lambda^n + a_{n-1}\lambda^{n-1} \dots + a_0$. Since $\varepsilon^2= 0$, we can compute that $P(\mu+\nu\varepsilon) = P(\mu) + \nu\varepsilon( na_n\mu^{n-1}+ (n-1)a_{n-1}\mu^{n-2}+\dots a_1) = P(\mu) + \nu\varepsilon P'(\mu)$, where $P'$ denotes the formal derivative of $P$. Since $P(\mu+\nu\varepsilon)$ has to be zero, this implies $\mu$ is a root for $P$ and its formal derivative. This implies $\mu$ is a root of $P$ of multiplicity at least $2$.
	\end{proof}
	\begin{proposition}\label{prop:derimpliessolid}
		Let $(A,X)$ be  a primitive axial algebra of Jordan type $\tfrac{1}{2}$ over a field $\F$. If $a,b\in X$ such that $D_{a,b}$ is a derivation, then $\subalg$ is solid.
	\end{proposition}
	\begin{proof}
		We can assume $(a,b)=\tfrac{1}{4}$ and $\subalg$ is $3$-dimensional, by \cref{thm:solid14}. We may assume $\F$ is algebraically closed, by \cref{def:solid} and the fact $D_{a,b}$ is still a derivation after base change. Since $(\tau_a\tau_b)^3 = \tau_{a^{\tau_b}}\tau_{b^{\tau_a}} = \id_A$, the automorphism $\rho=\tau_a\tau_b$ has order $3$.
		
		Clearly $\subalg$ is toric unless $\kar \F=3$, and then $\subalg$ is baric. We will set $\mathcal{P}_{x,y} = \lambda^2P^t_{x,y}, \mathcal{Q}_x = \lambda^2Q^t_x\in A[\lambda]$ if $\kar \F \neq 3$ and $\mathcal{P}_{x,y} = P^u_{x,y}, \mathcal{Q}_x = Q^u_x\in A[\lambda]$ if $\kar \F = 3$ for all $x,y\in A$, so $\mathcal{P}_{x,y},\mathcal{Q}_x$ are polynomial maps.
		
		By \cref{prop:dualderiv} and \eqref{eq:derivifaut}, $\chi = \id_{A_{\F[\varepsilon]}} - 8\varepsilon [L_{a}, L_v]$ is an automorphism of $ A\otimes_{\F}\F[\varepsilon]$, where $v=e-f$ if $\kar \F \neq 3$ and $v=2(ab-a)$ if $\kar \F =3$. This means that if $a_\lambda$ is an axis of $ A\otimes_{\F}\F[\varepsilon]$, then $a_\lambda^{\chi}$ is also an axis of $ A\otimes_{\F}\F[\varepsilon]$. By \cref{rem:miyactionR}, since $\chi$ restricted to $\subalg_R$ is  $\tau_{a}\tau_{a+\varepsilon v}$, we have $a_\lambda^{\chi}   = a_{\lambda(1+2\varepsilon)} $ if $\kar \F \neq 3$ and $a_\lambda^{\chi}=a_{\lambda+2\varepsilon}$ if $\kar \F =3$. 
		 
		 As $a_\lambda$ is an axis for at least $3$ different values of $\lambda \in \F$, the polynomials $\mathcal{P}_{x,y}, \mathcal{Q}_x\in A[\lambda]$ have at least 6 roots for every $x,y \in A$ by \cref{lem:polytocheck,lem:multiplicity2}. 
		  Since $\deg_\lambda \mathcal{P}_{x,y},\deg_\lambda \mathcal{Q}_x \leq 4$, this means $\mathcal{P}_{x,y}=\mathcal{Q}_x =0$ for all $x,y\in A$. By \cref{lem:polytocheck}, $\subalg$ is solid.
	\end{proof}
	Now \cref{cor:derall,prop:derimpliessolid} together show \cref{thm:solidiffderiv}.
	\section{Proof of \cref{thm:allsolidisjordan}}\label{sec:thm13}
	Next, we need almost Jordan rings. Such rings were first studied by Osborn \cite{identityfour}. 
	\begin{definition}
		A non-associative ring $R$ is \emph{almost Jordan} if 
		\begin{equation}\label{id4}
			2((yx) x) x + yx^3 = 3(yx^2)x \text{ for all } x,y\in R.
		\end{equation}
	\end{definition}
	
	\begin{lemma}\label{lem:deralmostjord}
		A non-associative commutative ring $R$ with $2\in R^{\times}$ is almost Jordan if and only if $D_{x,y}$ is a derivation of $R$ for all $x,y\in R$.
	\end{lemma}
	\begin{proof}
		This is proven in \cite[p.\@ 1115, Equations~(4) and (5)]{identityfour}.
	\end{proof}
	\begin{corollary}\label{cor:almostsolid}
		Suppose $(A,X)$ is  a primitive axial algebra of Jordan type $\tfrac{1}{2}$, and $X$ spans $A$ linearly. If for every $a,b \in X$ the line $\llangle a,b \rrangle$ is solid, then $A$ is almost Jordan.
	\end{corollary}
	\begin{proof}
		For every $a,b \in X$, $D_{a,b}$ is a derivation by \cref{thm:solidiffderiv}. Since $X$ spans $A$, the map $[L_x,L_y]$ is a derivation for every $x,y\in A$. \cref{lem:deralmostjord} shows $A$ is almost Jordan. 
	\end{proof}
	 As was noted in \cite{identityfour}, the existence of idempotents often forces almost Jordan algebras to be Jordan. We prove a similar result that is applicable to our situation.
	\begin{proposition}\label{prop:almostisjordan}
		An almost Jordan algebra $A$ over a field $\F$ with $\kar \F\neq 2$ that is linearly spanned by primitive idempotents is a Jordan algebra.
	\end{proposition}
	\begin{proof}
		By \cite[Lemma~1]{identityfour}, if $a\in A$ is any idempotent, it is a Jordan type $\tfrac{1}{2}$ axis. Since $A$ is spanned by idempotents, it suffices to prove the linearized Jordan identity (\cite[p.95]{char3})
		\begin{equation}\label{eq:jordan}
			(yz)(ax)+(xy)(az) +(xz)(ay)-((yz)a)x -((xy)a)z- ((xz)a)y =0
		\end{equation}
		for all $x,y,z,a\in A$, by \cite[p.95 and Lemma~3.1]{char3}. Let $J(a,x,y,z)$ denote the left-hand side of \eqref{eq:jordan}. Because $A$ is spanned by primitive axes, we may assume $a$ is a primitive axis, and $x,y,z$ are eigenvectors of $L_a$ with eigenvalues $\lambda_x,\lambda_y,\lambda_z\in \{1,0,\tfrac{1}{2}\}$.  Then 
		\begin{equation}\label{eq:jordaneigvalues}
			J(a,x,y,z) = \lambda_x(yz)x + \lambda_z(xy)z + \lambda_y(xz)y-((yz)a)x-((xy)a)z- ((xz)a)y.
		\end{equation}
		Because of the symmetric roles of $x$, $y$, and $z$, we need to check this is zero for $(\lambda_x,\lambda_y,\lambda_z) = (1,1,1)$, $(1,1,\tfrac{1}{2})$, $(1,1,0)$, $(1,\tfrac{1}{2},\tfrac{1}{2})$, $(1,\tfrac{1}{2},0)$, $(1,0,0)$, $(\tfrac{1}{2},\tfrac{1}{2},\tfrac{1}{2})$, $(\tfrac{1}{2},\tfrac{1}{2},0)$, $(\tfrac{1}{2},0,0)$ and $(0,0,0)$.
		
		Since $a$ is primitive, if $\lambda_x  =1$ we can assume $x=a$, and \eqref{eq:jordaneigvalues} becomes equal to $a(yz)-a(a(yz)) - (\lambda_y - \lambda_z)^2yz$. If moreover $\lambda_y=1$, then as above we may assume $y=a$, and \eqref{eq:jordaneigvalues} becomes equal to $a(az)-a(a(az)) - (1 - \lambda_z)^2az = -(2\lambda_z^3-3\lambda_z^2+\lambda_z)z$, which is zero for $\lambda_z=0,1$ or $\tfrac{1}{2}$. If $x=a,\lambda_y = \tfrac{1}{2},\lambda_z = 0$, then $J(a,x,y,z) =a(yz)-a(a(yz)) - \tfrac{1}{4}yz  =\tfrac{1}{2}yz- \tfrac{1}{4}yz - \tfrac{1}{4}yz =0 $, since $yz$ is a $\tfrac{1}{2}$-eigenvector by the Jordan fusion law, and if $x=a,\lambda_y = \tfrac{1}{2},\lambda_z =  \tfrac{1}{2}$ then $a(yz) = \mu a$ for a certain $\mu\in \F$, by the Jordan fusion law and primitivity of $a$, so $J(a,x,y,z)= \mu a - \mu a^2 - (\lambda_y - \lambda_z)^2yz = 0$.
		
		On the other hand, if $\lambda_z=0$, then $xz$ and $yz$ are (possibly zero) $\lambda_x$- and $\lambda_y$-eigenvectors for $L_a$, respectively, by the Jordan fusion law. Then $ J(a,x,y,z) = \lambda_x(yz)x + \lambda_y(xz)y-\lambda_y(yz)x-((xy)a)z- \lambda_x(xz)y$. If then also $\lambda_y=0$, this becomes, by the same argument as above, equal to $\lambda_x(yz)x - \lambda_x(xy)z -\lambda_x(xz)y$, which is zero if $x=a$ or $\lambda_x= 0$. If $\lambda_x = \lambda_y = \tfrac{1}{2}$, we get $J(a,x,y,z) = \tfrac{1}{2}(yz)x + 0+ \tfrac{1}{2}(xz)y-\tfrac{1}{2}(yz)x-((xy)a)z- \tfrac{1}{2}(xz)y = 0$, since in this case $((xy)a)z = 0$ by the Jordan fusion law.
		
		If $(\lambda_x,\lambda_y,\lambda_z)=(\tfrac{1}{2},\tfrac{1}{2},\tfrac{1}{2})$, then $D_{a,x}(yz) = D_{a,x}(y)z+D_{a,x}(z)y$, since $A$ is almost Jordan.
		Writing this out, we get, since $a(x(yz)) = \tfrac{1}{2}x(yz)$ by the Jordan fusion law,
		\begin{equation*}
			\tfrac{1}{2}(yz)x -(a(yz))x  =   (a(xy))z -\tfrac{1}{2}(xy)z+ (a(xz))y- \tfrac{1}{2}(xz)y . 
		\end{equation*}
		This means \eqref{eq:jordaneigvalues} is zero for $(\lambda_x,\lambda_y,\lambda_z) = (\tfrac{1}{2},\tfrac{1}{2},\tfrac{1}{2})$.
		If $(\lambda_x,\lambda_y,\lambda_z)=(\tfrac{1}{2},0,0)$, then $((yz)a)x=0$, $((xy)a)z=\tfrac{1}{2}(xy)z$ and $((xz)a)y=\tfrac{1}{2}(xy)z$ by the Jordan fusion law. So \eqref{eq:jordaneigvalues} becomes equal to $\tfrac{1}{2}(yz)x - \tfrac{1}{2}(xy)z -\tfrac{1}{2}(xz)y $,	which is zero by \cite[Lemma~2]{identityfour}.	
	\end{proof}
	\begin{remark}
		When $(A,X)$ has a nondegenerate Frobenius form and $\kar\F\neq 2,3,5$, there is a less computational method to prove this, see \cite[Proposition~A.8]{chayet2020class}.
	\end{remark}
	We now generalize \cref{prop:almostisjordan} for $\kar \F \neq 3$ to obtain \cref{thm:allsolidisjordan}, by using \cref{thm:solidiffderiv}.
	 The following lemma is adapted from a conversation with Sergey Shpectorov.
	\begin{lemma}\label{solid3gen}
		Let $(A,X)$ be a primitive axial $\F$-algebra of Jordan type $\tfrac{1}{2}$  with $\kar\F \neq 2,3$. If $a,b,c \in X$ are three  different axes such that $\llangle a,b \rrangle$ and $\llangle a,c \rrangle$ are solid, then $\llangle a,c^{\tau_b} \rrangle$ is solid. Moreover, if not all the values $(a,b)$, $(a,c)$ and $(a,bc)$ are zero, then $\llangle b,c \rrangle$ is solid.
	\end{lemma}
	\begin{proof}
		We may assume $\mathbb{F}$ is algebraically closed, as in \cref{prop:derimpliessolid}. We will begin by proving the second claim.
		 Suppose first that $\llangle a,b\rrangle$ or $\llangle a,c\rrangle$ is not flat, say $\llangle a,b\rrangle$ is not flat without loss of generality. We want to find $p_1,p_2\in \llangle a,c\rrangle$ such that $a,p_1,p_2$ span $\llangle a,c \rrangle$ linearly and $(a,p_i)\neq \tfrac{1}{4}$ for $i=1,2$. If $\llangle a,c \rrangle $ is flat, the choice $p_1= c$, $p_2 = c+ac$ satisfies the properties above by \cref{lem:flatfacts,lem:allprimflat}. If $\llangle a,c \rrangle $ is baric, we can choose $p_1 = a+v+v^2$, $p_2 = a+2v+4v^2$ by \cref{lem:baricfacts,lem:allprimbaric} and the fact that $1\neq \tfrac{1}{4}$ since $\kar \F \neq 3$. If $\llangle a,c \rrangle $ is toric, we have a basis $e_c,\oneb_c,f_c$ for $\llangle a,c \rrangle$ with $a = e_c+\tfrac{1}{2}u_c+f_c$ as in \cref{lem:toricfacts}. Then by \cref{lem:toricfacts,lem:allprimtoric}, we can choose $p_1 = -e_c+\tfrac{1}{2}\oneb_c -f_c$, $p_2 = \mu e_c + \tfrac{1}{2}\oneb_c -\mu f_c$, where $\mu$ is a root of $x^2+1 \in \F[x]$, since $(a,p_1) = 0$ and $(a,p_2) = \tfrac{1}{2}$.
		 Next, we want to find axes  $q_1,q_2 \in \subalg$ such that $a,q_1,q_2$ span $\llangle a,b \rrangle$ linearly and $(p_i,q_j) \neq \tfrac{1}{4}$ for $i,j\in \{1,2\}$. 
		Write $a_\lambda,\, \lambda \in \F$ for the axes in $\subalg$ as in \cref{lem:allprimtoric} or \cref{lem:allprimbaric} when $\subalg$ is toric or baric respectively. Then the maps $\pi_i\colon \F\to \mathbb{F} \colon \lambda \mapsto (p_i,a_\lambda)$ for $i=1,2$ are rational functions with respect to $\lambda$. Since $(p_i,a) \neq \tfrac{1}{4}$ for $i=1,2$, the maps $\pi_i$, $i =1,2$ are not equal to the constant function $\tfrac{1}{4}$. So the number of axes $q\in \subalg$ with $(q,p_1) = \tfrac{1}{4}$ or $(q,p_2) = \tfrac{1}{4}$ is finite, since the rational maps $\pi_i - \tfrac{1}{4}$, $i = 1,2$ have only finitely many zeroes. Thus we can choose $q_1\in \llangle a,b \rrangle$ with $(q_1,p_1),(q_1,p_2)\neq \tfrac{1}{4}$ and $\dim_\F\langle a, q_1\rangle=2$. If $\dim_\F\subalg =2$, set $q_1=q_2$. If $\dim_\F\subalg=3$, by \cref{lem:allprimtoric,lem:allprimbaric}, there are infinitely many $q_2\in\llangle a,b\rrangle$ with $\dim_\F\langle a, q_1, q_2\rangle=3$, so, by the above, we can find axes $q_1,q_2 \in \subalg$ such that $a,q_1,q_2$ span $\llangle a,b \rrangle$ and $(p_i,q_j) \neq \tfrac{1}{4}$ for $i,j\in \{1,2\}$. 
		
		Now we have that $[L_a,L_{q_i}],[L_a,L_{p_i}]$ and $[L_{p_i},L_{q_j}]$ are derivations of the algebra $A$ for $i,j\in \{1,2\}$ by \cref{thm:solid14,thm:solidiffderiv}. Since $a,p_1,p_2$ span $\llangle a,c\rrangle$ and $a,q_1,q_2$ span $\llangle a,b\rrangle$, we have that $[L_x,L_y]$ is a linear combination of $[L_{p_i},L_{q_j}],[L_a,L_{q_i}],[L_a,L_{p_i}]$ for every $x\in \llangle a,b\rrangle, y \in \llangle a,c \rrangle$, so $[L_b,L_c]$ is also a derivation and $\llangle b,c \rrangle$ is solid.
		 
		If both $\subalg$ and $\llangle a,c \rrangle$ are flat but $(a,bc) = (c,ab) \neq 0$, then by a similar reasoning there exist $\lambda, \mu \in \mathbb{F}^{\times}$ such that $\llangle a + \lambda ab ,c \rrangle ,  \llangle b + \mu ab ,c \rrangle$ are solid, so $[L_a,L_c]$, $[L_{a+\lambda a b},L_c]$, $[L_{b+\mu a b},L_c] $ are derivations. Then $[L_b,L_c]$ is a derivation, so $\llangle b,c \rrangle$ is solid by \cref{thm:solidiffderiv}.
	
		For the first claim, in both cases above, we have for any $x\in \llangle a,b \rrangle $ that $D_{x,c}$ is a derivation. Hence $D_{a^{\tau_b},c}$ is a derivation, so $D_{a,c^{\tau_b}}$ is as well, since $\tau_b \in \Aut(A)$. If $(a,b)=(a,c)=(a,bc) =0$, then by \cref{lem:basiccomp}, $(a,c^b) = (a,c + 4(b,c)b - 4bc) = 0$, so $\llangle a,b^c \rrangle $ is again solid by \cref{thm:solid14}.
	\end{proof}
	We are now finally able to prove the remainder of \cref{thm:allsolidisjordan}.
	\begin{theorem}
		A primitive axial algebra $(A,X)$ of Jordan type $ \tfrac{1}{2}$ over a field $\mathbb{F}$ with $\kar \mathbb{F} \neq 2,3$ such that for all $a,b\in X$ the line $\subalg$ is solid, is Jordan.
	\end{theorem}
	\begin{proof}
		By \cite[Corollary~1.2]{primjordan}, the set $\overline{X} = X^{\Miy(X)}$ spans the axial algebra $A$ linearly. If we prove that for all $a,b\in \overline{X}$  the subalgebra $\llangle a,b \rrangle$ is solid, the result follows from \cref{cor:almostsolid,prop:almostisjordan}. We prove that for $a,b\in X$ and $g,h\in \Miy(X)$ that $\llangle a^g,b^h\rrangle$ is solid by induction on the length of $gh^{-1}$ in terms of the generators $\tau_d$, with $d\in X$. Let $a,b,c\in X$. Then $\llangle a,b^{\tau_c} \rrangle$ is solid by \cref{solid3gen}. Then$\llangle a^{\tau_{c_1}\dots \tau_{c_k}} , b^{\tau_{c_{n}} \dots \tau_{c_{k+1}}} \rrangle$ is solid if and only if $\llangle a^{\tau_{c_1}\dots \tau_{c_n}} , b\rrangle$ is. By induction, $\llangle a^{\tau_{c_1}\dots \tau_{c_{n-1}}} , b\rrangle$ and $\llangle a^{\tau_{c_1}\dots \tau_{c_{n-1}}} , c_n\rrangle$ are solid, so by \cref{solid3gen}, so is $\llangle a^{\tau_{c_1}\dots \tau_{c_{n-1}}} , b^{\tau_{c_{n}}}\rrangle$. So  $\llangle a^{\tau_{c_1}\dots \tau_{c_{n}}} , b\rrangle $ is solid as well.
	\end{proof}

	\bibliographystyle{plain}
	\bibliography{sources.bib}	
\end{document}